\documentclass[12pt,notitlepage]{amsart}
\usepackage{latexsym,amsfonts,amssymb,amsmath,amsthm}
\usepackage{graphicx}
\usepackage{color}
\usepackage[normalem]{ulem}
\let\turc\c
\usepackage{etoolbox}
\robustify\turc %new thing to get the c in Topacogullari
\renewcommand{\c}{\mathfrak{c}}
\newcommand{\bpm}{\begin{pmatrix}}
\newcommand{\epm}{\end{pmatrix}}

\usepackage[inner=1.0in,outer=1.0in,bottom=1.0in, top=1.0in]{geometry}

\newcommand{\mz}{\ensuremath{\mathbb Z}}

\newcommand{\intR}{\int_{-\infty}^{\infty}}

\newcommand{\sumstar}{\sideset{}{^*}\sum}

\theoremstyle{plain}		
	\newtheorem{mytheo}{Theorem} [section]
	
	\newtheorem{myprop}[mytheo]{Proposition}
	
     \newtheorem{mylemma}[mytheo]{Lemma}

	\newtheorem{myexample}[mytheo]{Example}

\theoremstyle{remark}

\numberwithin{equation}{section}
\numberwithin{figure}{section}

\begin{document}

\author{Matthew P. Young}
 \address{Department of Mathematics \\
 	  Texas A\&M University \\
 	  College Station \\
 	  TX 77843-3368 \\
 		U.S.A.}

 \email{myoung@math.tamu.edu}
 \thanks{This material is based upon work supported by the National Science Foundation under agreement No. DMS-2001306.  Any opinions, findings and conclusions or recommendations expressed in this material are those of the authors and do not necessarily reflect the views of the National Science Foundation.  }

\begin{abstract} 
We prove an improved spectral large sieve inequality for the family of $SL_3(\mathbb{Z})$ Hecke-Maass cusp forms.
The method of proof uses duality and its structure reveals unexpected connections to Heath-Brown's large sieve for cubic characters.
\end{abstract}

 \title{An improved spectral large sieve inequality for $SL_3(\mathbb{Z})$}
\maketitle

\section{Introduction}
A large sieve inequality for a family of automorphic forms is a flexible and versatile tool that represents quantitative orthogonality properties of the family.  Strong results are known for $GL_1$ and $GL_2$ (e.g. see \cite{Montgomery} and \cite[Chapter 7]{IK} for some surveys), but progress has been more elusive in higher rank.  The
main focus of this article, the
$SL_3(\mathbb{Z})$ spectral large sieve, has seen some recent attention in a series of papers \cite{Blomer, Young, BlomerButtcane}.  For some other notable higher rank examples, see
\cite{DukeKowalski, Venkatesh, ThornerZaman}.

We set some notation before continuing the discussion on the large sieve.
Let $\mathcal{F}^{\text{cusp}}$ denote the family of Hecke-Maass cusp forms on $SL_3(\mathbb{Z})$.  
Similarly, let $\mathcal{F}^{\text{Eis}}$ denote the family of $SL_3(\mathbb{Z})$ Eisenstein series induced by $SL_2(\mathbb{Z})$ cusp forms (see \cite[Section 10.5]{Goldfeld} for a definition).  Let $\mathcal{F} = \mathcal{F}^{\text{cusp}} \cup \mathcal{F}^{\text{Eis}}$.
For $F \in \mathcal{F}$, let $\mu_F = \mu = (\mu_1, \mu_2, \mu_3) \in \mathfrak{a}_{\mathbb{C}}^*$ be its Langlands parameters, so the Ramanujan conjecture predicts that $\mu \in i \mathbb{R}^3$.  Let $\lambda_{F}(m,n)$ denote the Hecke eigenvalues of $F$.  Let $\Omega \subset \mathfrak{a}^*$ be compact, Weyl group invariant, and disjoint from the Weyl chamber walls.  
Let $B = B_{V}$ be a box of sidelength $V/100$, with $100 \leq V \leq T$ and $B_{V} \subset T \Omega$, and let $\mathcal{F}_{V} \subset \mathcal{F}$ denote the set of Hecke-Maass cusp forms and Eisenstein series with $\mu_F \in W(B)$, where $W$ is the Weyl group.  Write
\begin{equation}
\mathcal{F}_{V}^{\text{cusp}} = \mathcal{F}^{\text{cusp}} \cap \mathcal{F}_{V},
\qquad
\text{and}
\qquad
\mathcal{F}_{V}^{\text{Eis}} = \mathcal{F}^{\text{Eis}} \cap \mathcal{F}_{V}.
\end{equation}
 For $F \in \mathcal{F}^{\text{cusp}}$, let $\omega_F = \text{Res}_{s=1} L(F \otimes \overline{F}, s)$.  The Weyl law proved by Lapid and M\"{u}ller \cite{LM} gives that the cardinality of $\mathcal{F}_V^{\text{cusp}}$ is $V^2 T^{3+o(1)}$.
 Any potential violation to the Ramanujan conjecture must occur near the Weyl chamber walls (e.g. see \cite[p.678]{Blomer}), so automatically any cusp form with $\mu_F \in T \Omega$ satisfies Ramanujan.
For ${\bf a} \in \ell^2$, we use the notation $|{\bf a}| = \| {\bf a}\|_2$.

The culmination of \cite{Blomer, Young, BlomerButtcane} is the following.
\begin{mytheo}[\cite{BlomerButtcane}] 
\label{thm:BlomerButtcane}
We have
\begin{equation}
\label{eq:BlomerButtcaneWide}
\sum_{* \in \{ \text{cusp}, \text{Eis} \}}
\sum_{F \in \mathcal{F}_T^{*}} \frac{1}{\omega_F} 
\Big|
\sum_{N \leq n \leq 2N} a_n \lambda_{F}(1,n) \Big|^2
\ll (T^5 + T^2 N)^{1+\varepsilon} |{\bf a} |^2,
\end{equation}
\begin{equation}
\label{eq:BlomerButtcaneLocal}
\sum_{* \in \{ \text{cusp}, \text{Eis} \}}
\sum_{F \in \mathcal{F}_1^{*}} \frac{1}{\omega_F} 
\Big|
\sum_{N \leq n \leq 2N} a_n \lambda_{F}(1,n) \Big|^2
\ll (T^3 + T^2 N)^{1+\varepsilon} |{\bf a}|^2.
\end{equation}
\end{mytheo}
The formulations of \eqref{eq:BlomerButtcaneWide} and \eqref{eq:BlomerButtcaneLocal} are imprecise because we have not fully described the meaning of $\sum_{F \in \mathcal{F}^{\text{Eis}}} \omega_F^{-1}$ for $F$ an Eisenstein series; see e.g. %\cite[Section 3.1]{BlomerButtcaneSubconvexity} or 
\cite[Section 4]{BlomerButtcane} for the correct normalizing factor.
 The proof of Theorem \ref{thm:BlomerButtcane} notably relies on the $GL_3$ Kuznetsov formula.                                                                                                                                                                                                                                                                             The spectral side of the Kuznetsov formula includes both the cusp forms as well as Eisenstein series, which explains why Theorem \ref{thm:BlomerButtcane} covers both types of automorphic forms.
 %Because of this feature, Blomer and Buttcane actually prove a stronger result where the left hand sides of \eqref{eq:BlomerButtcaneWide} and \eqref{eq:BlomerButtcaneLocal} also contain complementary contributions from the Eisenstein series.

By general principles of bilinear forms (cf. \cite[Chapter 7]{IK}), the optimal bound one could expect on the right hand side of \eqref{eq:BlomerButtcaneWide} would be $(T^5 + N)^{1+\varepsilon}$, while that of
\eqref{eq:BlomerButtcaneLocal} would be $(T^3 + N)^{1+\varepsilon}$.  However, Blomer and Buttcane showed that the term $T^2 N$ in \eqref{eq:BlomerButtcaneWide} cannot be reduced in size, by constructing a choice of vector ${\bf a}$ so that the contribution from the Eisenstein series is at least $T^2 N |{\bf a}|^2$, for $N \gg T^{3+\delta}$.   An examination of the proof of \cite[Proposition 1.3]{BlomerButtcane}
shows their method leads to a lower bound of size $TN |{\bf a }|^2$ for the left hand side of \eqref{eq:BlomerButtcaneLocal}.
In Section \ref{section:lowerbound} we sketch an alternative method to produce this lower bound.

A natural question is if the bounds \eqref{eq:BlomerButtcaneWide}-\eqref{eq:BlomerButtcaneLocal} can be improved when the family is restricted to cusp forms.
The main result of this article affirms this.
\begin{mytheo}
\label{thm:mainthm}
We have
%\begin{equation}
%\label{eq:DualLargeSieveWide}
%\sum_{F \in \mathcal{F}_T} \frac{1}{\omega_F} 
%\Big|
%\sum_{N \leq n \leq 2N} a_n \lambda_{F}(1,n) \Big|^2
%\ll (??)^{1+\varepsilon} \|{\bf a} \|_2^2,
%\end{equation}
\begin{equation}
\label{eq:DualLargeSieveLocal}
\sum_{F \in \mathcal{F}_1^{\text{cusp}}} \frac{1}{\omega_F} 
\Big|
\sum_{N \leq n \leq 2N} a_n \lambda_{F}(1,n) \Big|^2
\ll (T^5 + N + T^2 N^{2/3})^{1+\varepsilon} |{\bf a}|^2.
\end{equation}
\end{mytheo}
Note that the right hand side of \eqref{eq:DualLargeSieveLocal} is smaller than the right hand side of \eqref{eq:BlomerButtcaneLocal} for $N \gg T^{3+\varepsilon}$, and is also just as good as the ``$N + T^3$" theoretically optimal bound for $N \gg T^{6}$.

The starting point of our proof is to use the duality principle and the functional equation of Rankin-Selberg $L$-functions on $GL_3 \times GL_3$.  This method is most effective when $N$ is large, since this makes the dual length of summation relatively shorter.  The final step in our proof is an application of the Buttcane-Blomer bound \eqref{eq:BlomerButtcaneLocal}, which is strongest for relatively small values of $N$.

A curious aspect of the proof is that is reveals that certain aspects of the family $\mathcal{F}$ are in analogy with the family of cubic Hecke characters.  A large sieve inequality for this latter family was proved by Heath-Brown \cite{HeathBrownCubicSieve} with an application to  the problem of estimating sums of cubic Gauss sums of prime arguments.  See Section \ref{section:cubic} below for a more thorough discussion of Heath-Brown's work and its connections to our proof of Theorem \ref{thm:mainthm}.

For simplicity, Theorem \ref{thm:mainthm} is stated for the localized family $\mathcal{F}_1$ but in principle one could use the same method to study $\mathcal{F}_T$ as well.  Typically, small families are more difficult than large families, so one might expect that a bound on $\mathcal{F}_T$ would be even easier to prove than that for $\mathcal{F}_1$.  However, our proof of Theorem \ref{thm:mainthm} using duality requires the conductor of $L(1/2, F \otimes \overline{G})$ for $F, G \in \mathcal{F}$ which is a bit simpler to express for $F, G \in \mathcal{F}_1$ than for general $F, G \in \mathcal{F}_T$.  Generically, for $F, G \in \mathcal{F}_T$, the conductor of 
$L(1/2, F \otimes \overline{G})$ is of size $T^{9}$ but there are various conductor-dropping ranges to consider.  Indeed, for $F, G \in \mathcal{F}_1$, the conductor of $L(1/2, F \otimes \overline{G})$ is of size $T^{6}$.  As an aside, this discussion indicates that the approach via duality is beneficial when $N \gg T^{3+\delta}$ (since $T^3$ is the square-root of the conductor), consistent with the remark above that \eqref{eq:DualLargeSieveLocal} is an improvement in this range.

\section{Preliminaries}
\subsection{Maass forms on $SL_3(\mathbb{Z})$}
\begin{mylemma}[Hecke relations]
\label{lemma:Hecke}
Let $F \in \mathcal{F}$.  Then
\begin{equation}
\label{eq:HeckeRelation0}
\lambda_F(m,1) \lambda_F(1,n) = \sum_{d|(m,n)} \lambda_F\Big(\frac{m}{d}, \frac{n}{d}\Big),
\end{equation}
and
\begin{equation}
\label{eq:HeckeRelation}
\lambda_F(m,n) =  \sum_{d|(m,n)} \mu(d) \lambda_F\Big(\frac{m}{d}, 1\Big) \lambda_F\Big(1, \frac{n}{d}\Big).
\end{equation}
Moreover,
\begin{equation}
\label{eq:HeckeDuality}
\lambda_F(m,n) = \overline{\lambda_F}(n,m).
\end{equation}
\end{mylemma}
The relation \eqref{eq:HeckeRelation0} appears in \cite[Theorem 6.4.11]{Goldfeld},  from which \eqref{eq:HeckeRelation} follows from M\"obius inversion.  For \eqref{eq:HeckeDuality}, see \cite[Theorem 9.3.11 Addendum]{Goldfeld}.

\begin{mylemma}[Convexity bound]
\label{lemma:convexity}
For any $F \in \mathcal{F}_T$ and any $X \geq 1$ we have
\begin{equation}
\label{eq:convexity}
\sum_{m^2 n \leq X} |\lambda_F(m,n)|^2 \ll_{\varepsilon} X (XT)^{\varepsilon}.
\end{equation}
\end{mylemma}
This follows from work of Xiannan Li \cite{XiannanLi}.

\begin{mylemma}
Let $F, G \in \mathcal{F}^{\text{cusp}}$.  The Rankin-Selberg $L$-function $L(s, F \otimes \overline{G})$ is defined by
\begin{equation}
\label{eq:RankinSelbergDefinition}
L(s, F \otimes \overline{G}) = \sum_{d, m, n \geq 1} \frac{\lambda_F(m,n) \overline{\lambda}_G(m,n)}{(d^3 m^2 n)^s}.
\end{equation}
It has meromorphic continuation to $s \in \mathbb{C}$ with a possible pole at $s=1$ only, and
satisfies the functional equation
\begin{equation}
\label{eq:RankinSelbergFunctionalEquation}
\gamma(s, F \otimes \overline{G}) 
L(s, F \otimes \overline{G})
=
\gamma(1-s, G \otimes \overline{F}) 
L(1-s, G \otimes \overline{F}),
\end{equation}
where
\begin{equation}
\label{eq:gammafactordef}
\gamma(s, F \otimes \overline{G}) 
= \gamma(s, \mu_F, \mu_G) 
=
\prod_{i,j=1}^{3} \Gamma_{\mathbb{R}}(s + \mu_i(F) + \overline{\mu_j}(G)).
\end{equation}
The pole at $s=1$ exists if and only if $F=G$.
\end{mylemma}
For a reference, see \cite[Theorem 7.4.9, Proposition 11.6.17]{Goldfeld}.

\subsection{Separation of variables}
\begin{mylemma}
\label{lemma:Fourier}
Suppose $f$ is Schwartz-class.  Then
\begin{equation}
f(x) = \intR \widehat{f}(y) e(xy) dy,
\end{equation}
where $\|\widehat{f}\|_1 \ll \|f \|_1 + \| f'' \|_1$.  The implied constant is absolute.
\end{mylemma}
This follows from Fourier inversion and integration by parts.  To give an idea of how we wish to use Lemma \ref{lemma:Fourier} to separate variables, consider the following example.
\begin{myexample}
\label{example:separationofvariables}
Suppose that %$(b_m) = {\bf b}$ is some sequence of complex numbers of finite support, and suppose 
$f$ is Schwartz-class.  Moreover suppose that $\gamma_m$ and $\delta_n$ are some sequences of real numbers and that $I$ is some finite set of integers.
Then
\begin{equation}
\label{eq:exampleformula}
\max_{|{\bf b}| = 1}
\Big| \sum_{m,n \in I} b_m \overline{b_n} f(\gamma_m + \delta_n) \Big|
\leq \|\widehat{f}\|_1 \max_{|{\bf b}| = 1} \Big|\sum_{m \in I} b_m \Big|^2.
\end{equation}
where $|{\bf b}| = (\sum_{n \in I} |b_n|^2)^{1/2}$.
\end{myexample}
\begin{proof}
By Lemma \ref{lemma:Fourier}, we have
\begin{align}
\Big| \sum_{m,n} b_m \overline{b_n} f(\gamma_m + \delta_n) \Big|
&=
\Big| \intR \widehat{f}(y)
\Big(\sum_m b_m e(\gamma_m y) \Big)
\Big(\sum_n \overline{b_n} e(\delta_n y) \Big)
\Big| dy
\\
& \leq 
\| \widehat{f} \|_1 
\max_{y \in \mathbb{R}}
\Big|\sum_m b_m e(\gamma_m y) \Big|
\Big|\sum_n \overline{b_n} e(\delta_n y) \Big|.
\end{align}
Taking the maximum over $|{\bf b}| = 1$ immediately gives the result.
\end{proof}

\section{Definitions of norms and some relations between them}
\label{section:combinatorial}
We begin by defining the basic norm that appears (implicitly) in Theorem \ref{thm:mainthm}:
\begin{equation}
\Delta_1(\mathcal{F}_V, N) 
=
\max_{|{\bf a}| = 1} 
\sum_{F \in \mathcal{F}_V^{\text{cusp}}} \frac{1}{\omega_F}
\Big|
\sum_{N \leq n \leq 2N} a_n \lambda_F(1,n)\Big|^2.
\end{equation}
By the duality principle (cf. \cite[p.170]{IK}), we have $\Delta_1(\mathcal{F}_{V}, N)  = \Delta^{(1)}(\mathcal{F}_{V}, N)$, where
\begin{equation}
\label{eq:Delta1DefDual}
\Delta^{(1)}(\mathcal{F}_{V}, N)
=
\max_{|{\bf b}| = 1} 
\sum_{N \leq n \leq 2N}
\Big|
\sum_{F \in \mathcal{F}_{V}^{\text{cusp}}}
 b_F \omega_F^{-1/2} \lambda_F(1,n)\Big|^2.
\end{equation}
We also define a related norm $\Delta_2(\mathcal{F}_{V}, N) = \Delta^{(2)}(\mathcal{F}_{V}, N)$ by
\begin{equation}
\Delta_2(\mathcal{F}_{V}, N) 
=
\max_{|{\bf a}| = 1} 
\sum_{F \in \mathcal{F}_{V}^{\text{cusp}}} \frac{1}{\omega_F}
\Big|
\sum_{N \leq m^2 n \leq 2N} a_{m,n} \lambda_F(m,n)\Big|^2,
\end{equation}
and where 
\begin{equation}
\label{eq:Delta2DefDual}
\Delta^{(2)}(\mathcal{F}_{V}, N)
=
\max_{|{\bf b}| = 1} 
\sum_{N \leq m^2 n \leq 2N}
\Big|
\sum_{F \in \mathcal{F}_{V}^{\text{cusp}}}
 b_F \omega_F^{-1/2} \lambda_F(m,n)\Big|^2.
\end{equation}
Finally we define a third norm $\Delta_3(\mathcal{F}_{V}, N) = \Delta^{(3)}(\mathcal{F}_{V}, N)$ by
\begin{equation}
\Delta_3(\mathcal{F}_{V}, N) 
=
\max_{|{\bf a}| = 1} 
\sum_{F \in \mathcal{F}_{V}^{\text{cusp}}} \frac{1}{\omega_F}
\Big|
\sum_{N \leq d^3 m^2 n \leq 2N} a_{d,m,n} \lambda_F(m,n)\Big|^2,
\end{equation}
and where 
\begin{equation}
\label{eq:Delta3DefDual}
\Delta^{(3)}(\mathcal{F}_{V}, N)
=
\max_{|{\bf b}| = 1} 
\sum_{N \leq d^3 m^2 n \leq 2N}
\Big|
\sum_{F \in \mathcal{F}_{V}^{\text{cusp}}}
 b_F \omega_F^{-1/2} \lambda_F(m,n)\Big|^2.
\end{equation}

We obviously have $\Delta_1(\mathcal{F}_{V}, N) \leq \Delta_2(\mathcal{F}_{V}, N) \leq \Delta_3(\mathcal{F}_{V}, N)$.  We also want relations in the other direction.
\begin{mylemma}
\label{lemma:Delta3intermsofDelta2}
We have
\begin{equation}
\label{eq:Delta3intermsofDelta2}
\Delta_3(\mathcal{F}_{V}, N) \ll
(\log N)
\max_{R \ll N }
\Big(\frac{N}{R}\Big)^{1/3} 
\Delta_2(\mathcal{F}_{V}, R).
\end{equation}
\end{mylemma}
\begin{proof}
We prove this on the dual side, using \eqref{eq:Delta3DefDual} and \eqref{eq:Delta2DefDual}.  By breaking the sum up so $R \leq m^2 n \leq 2R$ and summing $R$ over dyadic segments, we obtain
\begin{equation}
 \Delta^{(3)}(\mathcal{F}_{V}, N)
 \ll (\log N)
 \max_{1 \ll R \ll N} \Big(\frac{N}{R}\Big)^{1/3}
 \max_{|{\bf b}| = 1} 
\sum_{R \leq m^2 n \leq 2R}
\Big|
\sum_{F \in \mathcal{F}_{V}^{\text{cusp}}}
 b_F \omega_F^{-1/2} \lambda_F(m,n)\Big|^2.
\end{equation}
The result follows immediately.
%  Begin by restricting the inner sum over $d,m,n$ to $R \leq m^2 n \leq 2R$, whereby the size of $d$ is determined by $d \asymp (\frac{N}{R})^{1/3}$.  The number of such dyadic segments is $O(\log N)$.  
% Taking the sum over $d$ to the outside by Cauchy's inequality, we derive
% \begin{equation}
% \label{eq:Delta3intermsofDelta2firstEq}
% \Delta_3(\mathcal{F}_{V}, N) 
% \ll 
% N^{\varepsilon}
% \max_{R \ll N }
% \Big(\frac{N}{R}\Big)^{1/3} 
% \max_{|{\bf a}| = 1}
% \sum_d
% \sum_{F \in \mathcal{F}_{V}^{\text{cusp}}} \frac{1}{\omega_F}
% \Big|
% \sum_{R \leq m^2 n \leq 2R} a_{d,m,n} \lambda_F(m,n)\Big|^2.
% \end{equation}
% For a fixed $d$, set $a'_{m,n} = a_{d,m,n}$.
% Then we note
% \begin{equation}
% \label{eq:Delta3intermsofDelta2penultimateEq}
% \sum_d \sum_{F \in \mathcal{F}_{V}^{\text{cusp}}} \frac{1}{\omega_F}
% \Big|
% \sum_{R \leq m^2 n \leq 2R} a_{m,n}' \lambda_F(m,n)\Big|^2
% \leq  
% \sum_d \Delta_2(\mathcal{F}_{V}, R)
% \sum_{m,n} |a_{m,n}'|^2
% = \Delta_2(\mathcal{F}_{V}, R) |{\bf a}|^2
% .
% \end{equation}
% Inserting \eqref{eq:Delta3intermsofDelta2penultimateEq} into 
% \eqref{eq:Delta3intermsofDelta2firstEq} completes the proof.
\end{proof}

\begin{mylemma}
\label{lemma:Delta2intermsofDelta1}
We have
\begin{equation}
\label{eq:Delta2intermsofDelta1}
\Delta_2(\mathcal{F}_{V}, N) \ll
(NT)^{\varepsilon}
\max_{Y^2 X \ll N }
\min\Big( Y \Delta_1(\mathcal{F}_{V}, X), X \Delta_1(\mathcal{F}_{V}, Y) \Big).
\end{equation}
\end{mylemma}
\begin{proof}
Again, we prove this on the dual side, using \eqref{eq:Delta2DefDual} and \eqref{eq:Delta1DefDual}.
 By the Hecke relation \eqref{eq:HeckeRelation}, we deduce
  \begin{equation}
\Delta^{(2)}(\mathcal{F}_{V}, N)
=
\max_{|{\bf b}| = 1}
\sum_{N \leq m^2 n \leq 2N}
\Big|
\sum_{d|(m,n)}  \mu(d)
\sum_{F \in \mathcal{F}_{V}^{\text{cusp}}}
 b_F \omega_F^{-1/2} \lambda_F\Big(\frac{m}{d},1 \Big) \lambda_F\Big(1,\frac{n}{d}\Big) \Big|^2.
\end{equation} 
Applying Cauchy's inequality and a divisor function bound to take the sum over $d$ to the outside,
 we deduce
 \begin{equation}
\Delta^{(2)}(\mathcal{F}_{V}, N)
\ll N^{\varepsilon}
\max_{|{\bf b}| = 1}
\sum_{N \leq m^2 n \leq 2N}
\sum_{d|(m,n)} 
\Big|
\sum_{F \in \mathcal{F}_{V}^{\text{cusp}}}
 b_F \omega_F^{-1/2} \lambda_F\Big(\frac{m}{d},1 \Big) \lambda_F\Big(1,\frac{n}{d}\Big) \Big|^2.
\end{equation} 
Interchanging the order of summation and changing variables $m \rightarrow dm$ and $n \rightarrow dn$, we obtain
 \begin{equation}
\Delta^{(2)}(\mathcal{F}_{V}, N)
\ll N^{\varepsilon}
\max_{|{\bf b}| = 1}
\sum_{N \leq d^3 m^2 n \leq 2N}
\Big|
\sum_{F \in \mathcal{F}_{V}^{\text{cusp}}}
 b_F \omega_F^{-1/2} \lambda_F(m,1) \lambda_F(1,n) \Big|^2.
\end{equation}
Now we further restrict $d$, $m$ and $n$ so $d m \asymp Y$ and $n \asymp X$, and let $b_F' = b_F \lambda_F(m,1)$.  Then
\begin{align*}
\Delta^{(2)}(\mathcal{F}_{V}, N)
&\ll N^{\varepsilon}
\max_{XY^2 \ll N}
\max_{|{\bf b}| = 1}
\sum_{\substack{N \leq d^3 m^2 n \leq 2N \\ dm \asymp Y \\ n \asymp X}}
\Big|
\sum_{F \in \mathcal{F}_{V}^{\text{cusp}}}
 b_F' \omega_F^{-1/2} \lambda_F(1,n) \Big|^2
 \\
 &\ll N^{\varepsilon}
\max_{XY^2 \ll N} \max_{|{\bf b}| = 1}
\sum_{dm \asymp Y} \Delta^{(1)}(\mathcal{F}_V, X)
\sum_{F \in \mathcal{F}_{V}^{\text{cusp}}} |b_F|^2 |\lambda_F(m,1)|^2
.
\end{align*}
From Lemma \ref{lemma:convexity} we deduce
$
\sum_{dm \asymp Y} |\lambda_F(m,1)|^2 \ll Y (NT)^{\varepsilon},
$
uniformly in $F$,
leading to $$\Delta^{(2)}(\mathcal{F}_V, N) \ll (NT)^{\varepsilon} \max_{Y^2 X \ll N} Y \Delta^{(1)}(\mathcal{F}_V, X).$$

It remains to show that a similar bound holds but with $Y \Delta^{(1)}(\mathcal{F}_{V}, X)$ replaced by 
$X \Delta^{(1)}(\mathcal{F}_{V}, Y)$.  This follows by going through the same proof but reversing the roles of $m$ and $n$, and using \eqref{eq:HeckeDuality} along the way.
\end{proof}

Chaining together Lemmas \ref{lemma:Delta3intermsofDelta2} and \ref{lemma:Delta2intermsofDelta1}, we immediately deduce the following.
\begin{mylemma}
\label{lemma:Delta3intermsofDelta1}
We have
\begin{equation}
\label{eq:Delta3intermsofDelta1}
\Delta_3(\mathcal{F}_{V}, N) \ll
(NT)^{\varepsilon}
\max_{Y^2 X \ll N }
\Big(\frac{N}{XY^2}\Big)^{1/3} 
\min\Big( Y \Delta_1(\mathcal{F}_{V}, X), X \Delta_1(\mathcal{F}_{V}, Y) \Big).
\end{equation}
\end{mylemma}
See Section \ref{section:cubic} for a comparison of Lemma \ref{lemma:Delta3intermsofDelta1} with \cite[Lemma 6]{HeathBrownCubicSieve}.  

We also observe that the analogs of Lemmas \ref{lemma:Delta3intermsofDelta2}\textendash\ref{lemma:Delta3intermsofDelta1} hold equally well for Eisenstein series.

\section{Functional equation}
In this section we use the functional equation of the Rankin-Selberg $L$-function to deduce the following estimate.
\begin{mylemma}
\label{lemma:FunctionalEquation}
We have
\begin{equation}
\label{eq:FunctionalEquation}
\Delta^{(3)}(\mathcal{F}_1, N) \ll N +  \frac{N}{T^3} (NT)^{\varepsilon}
\max_{1 \leq Z \ll \frac{T^6}{N} (TN)^{\varepsilon}}
\Delta^{(3)}\Big(\mathcal{F}_1, Z \Big).
\end{equation}
\end{mylemma}
The proof of Lemma \ref{lemma:FunctionalEquation} crucially uses that the family is restricted to cusp forms.  The reader may examine the proof of Proposition \ref{prop:lowerbound} below to see how the family of Eisenstein series exhibits different behavior than the cusp forms.

\begin{proof}
Select a smooth nonnegative bump function $w$ with compact support on the positive reals, satisfying $w(x) \geq 1$ for $1 \leq x \leq 2$.  Then
\begin{equation}
\Delta^{(3)}(\mathcal{F}_{V}, N)
\leq 
\max_{|{\bf b}| = 1} 
\sum_{d,m,n} w\Big(\frac{d^3 m^2 n}{N}\Big)
\Big|
\sum_{F \in \mathcal{F}_{V}^{\text{cusp}}}
 b_F \omega_F^{-1/2} \lambda_F(m,n)\Big|^2.
\end{equation}
Next open the square, apply Mellin inversion, and evaluate the resulting Dirichlet series using \eqref{eq:RankinSelbergDefinition}, giving
\begin{equation}
\Delta^{(3)}(\mathcal{F}_{V}, N)
\leq 
\max_{|{\bf b}| = 1} 
\sum_{F, G \in \mathcal{F}_{V}^{\text{cusp}}}
 \frac{b_F \overline{b_G}}{\omega_F^{1/2} \omega_G^{1/2}} 
\frac{1}{2 \pi i} \int_{(3/2)} N^s \widetilde{w}(s) L(s, F \otimes \overline{G}) ds.
\end{equation}
Next we shift the contour of integration to the line $\text{Re}(s) = - \varepsilon$, change variables $s \rightarrow 1-s$, and apply the functional equation \eqref{eq:RankinSelbergFunctionalEquation}.  In this process we cross a potential pole at $s=1$ only, which exists if and only if $F=G$.  This pole contributes the term of size $O(N)$ to the right hand side of \eqref{eq:FunctionalEquation}.  In all we obtain $\Delta^{(3)}(\mathcal{F}_{V}, N)$ is at most $O(N)$ plus
\begin{equation}
\label{eq:afterFunctionalEquation}
\max_{|{\bf b}| = 1} 
\Big|
\sum_{F, G \in \mathcal{F}_{V}^{\text{cusp}}}
 \frac{b_F \overline{b_G}}{\omega_F^{1/2} \omega_G^{1/2}}
\frac{1}{2 \pi i} \int_{(3/2)} N^{1-s} \widetilde{w}(1-s)
\frac{\gamma(s, \mu_G, \mu_F )}{\gamma(1-s, \mu_F, \mu_G) }
 L(s, G \otimes \overline{F}) ds
 \Big|.
\end{equation}

Now we examine the ratio of gamma factors appearing in \eqref{eq:afterFunctionalEquation}.  Six out of the nine gamma factors in \eqref{eq:gammafactordef} have $|\mu_i(F) + \overline{\mu_j}(G)|$ large, of size $T$ (the precise size determined up to $O(1)$ by the location of the box $B$).  The remaining three gamma factors have $|\mu_i(F) + \overline{\mu_j}(G)|$ of size $O(1)$.
%; say these occur with $i=j$.  Then by Stirling we obtain
%\begin{equation}
%\frac{\gamma(s, G \otimes \overline{F})}{\gamma(1-s, F \otimes \overline{G}) }
%= Q^{s-\frac12} 
%\prod_{j=1}^{3} \frac{\Gamma_{\mathbb{R}}(s + \mu_j(F) + \overline{\mu_j}(G))}{\Gamma_{\mathbb{R}}(1-s + \mu_j(F) + \overline{\mu_j}(G))},
%\end{equation}
%where $Q = T^6$.
Moreover, since $F$ and $G$ automatically satisfy Ramanujan by the location of the box $B$, then $\mu_i(F) + \overline{\mu_j}(G) \in i \mathbb{R}$.
This means that for $\text{Re}(s) > 0$, we have that the ratio of gamma factors appearing in \eqref{eq:afterFunctionalEquation} is analytic, and satisfies the bound
\begin{equation}
\Big| Q^{\frac12 -s} \widetilde{w}(1-s) \frac{\gamma(s, \mu_G, \mu_F )}{\gamma(1-s, \mu_F, \mu_G) } \Big| 
\ll_{\text{Re}(s),A} (1+|s|)^{-A},
\end{equation}
for any $A > 0$, where $Q = T^6$.  Now in \eqref{eq:afterFunctionalEquation} we open up the Dirichlet series, obtaining
\begin{equation}
\label{eq:afterFunctionalEquation2}
\max_{|{\bf b}| = 1} 
\Big|
\sum_{d,m,n}
\sum_{F, G \in \mathcal{F}_{V}^{\text{cusp}}}
 \frac{b_F \overline{b_G} N}{ \omega_F^{1/2} \omega_G^{1/2} }
\frac{1}{2 \pi i} \int_{(3/2)}  \widetilde{w}(1-s)
\frac{\gamma(s, \mu_G, \mu_F )}{\gamma(1-s, \mu_F, \mu_G) }
 \frac{\lambda_G(m,n) \overline{\lambda_F}(m,n)}{(d^3 m^2 n N)^s} ds
 \Big|.
\end{equation}
We may truncate the Dirichlet series at $d^3 m^2 n \ll \frac{Q}{N} (TN)^{\varepsilon}$ with a very small error term (certainly smaller than the $O(N)$ term already accounted for), by shifting contours far to the right.  Having imposed this truncation, we may then shift the contour to the line $\text{Re}(s) = \varepsilon$.  
We may also truncate the integral at $|\text{Im}(s)| \ll (NT)^{\varepsilon}$, without producing a new error term.

Then \eqref{eq:afterFunctionalEquation2} is reduced to 
\begin{multline}
\label{eq:afterFunctionalEquation3}
\max_{|{\bf b}| = 1} 
\Big|
\sum_{d^3 m^2 n \ll \frac{T^6}{N} (NT)^{\varepsilon}}
\sum_{F, G \in \mathcal{F}_{V}^{\text{cusp}}}
 \frac{b_F \overline{b_G} N}{ \omega_F^{1/2} \omega_G^{1/2} }
 \\
\frac{1}{2 \pi i} \int_{\substack{\text{Re}(s) = \varepsilon \\ |\text{Im}(s)| \ll (NT)^{\varepsilon}}}  \widetilde{w}(1-s)
\frac{\gamma(s, \mu_G, \mu_F )}{\gamma(1-s, \mu_F, \mu_G) }
 \frac{\lambda_G(m,n) \overline{\lambda_F}(m,n)}{(d^3 m^2 n N)^s} ds
 \Big|.
\end{multline}

At a first pass, the reader is encouraged to ``pretend" that $\frac{\gamma(s, \mu_G, \mu_F )}{\gamma(1-s, \mu_F, \mu_G) }$ equals $Q^{s-\frac12}$ (which is a good first-order approximation) and continue with \eqref{eq:finalbound} to finish the proof.  Unfortunately, a rigorous argument is a bit more technical.  
The plan is to separate the variables $\mu_F$ and $\mu_G$ in the ratio of gamma factors.   The basic idea is encoded in Example \ref{example:separationofvariables}.
To this end, let $\mu_i(B)$, $i=1,2,3$ denote a point inside the box $B$ (the choice of point is irrelevant).  Then
\begin{equation}
\frac{\Gamma_{\mathbb{R}}(s + \mu_i(F) + \overline{\mu_j}(G))}{\Gamma_{\mathbb{R}}(1-s + \mu_i(F) + \overline{\mu_j}(G))}
=
\frac{\Gamma_{\mathbb{R}}(s + \mu_i(B) + \overline{\mu_j}(B) + i(\delta_i + \nu_j))}{\Gamma_{\mathbb{R}}(1-s + \mu_i(B) + \overline{\mu_j}(B) + i(\delta_i + \nu_j))},
\end{equation}
where $i\delta_i = \mu_i(F) - \mu_i(B)$ and $i\nu_j = \overline{\mu_j(G)} - \overline{\mu_j(B)}$.  Here $\delta_i, \nu_j = O(1)$ and are real.  The goal is to separate $\delta_i$ from $\nu_j$.  Let 
\begin{equation}
f(x) = 
\frac{\Gamma_{\mathbb{R}}(s + \mu_i(B) + \overline{\mu_j}(B) + ix)}
{\Gamma_{\mathbb{R}}(s + \mu_i(B) + \overline{\mu_j}(B))}
\frac{\Gamma_{\mathbb{R}}(1-s + \mu_i(B) + \overline{\mu_j}(B))}{\Gamma_{\mathbb{R}}(1-s + \mu_i(B) + \overline{\mu_j}(B) + ix)}.
\end{equation}
By Stirling, for $x\in \mathbb{R}$ and $|x| \ll 1$, we have $|f(x)| + |f''(x)| \ll T^{\varepsilon}$.  By Lemma \ref{lemma:Fourier} and \eqref{eq:exampleformula}, in effect this means we can separate the variables $\delta_i, \nu_j$ at ``cost'' at most $T^{\varepsilon}$.
Applying this with each of the gamma factors, we obtain that \eqref{eq:afterFunctionalEquation3} is bounded by
\begin{multline}
\label{eq:finalbound}
\frac{N}{T^3} (NT)^{\varepsilon} \max_{|{\bf b}| = 1} 
\sum_{d^3 m^2 n \ll \frac{T^6}{N} (NT)^{\varepsilon}}
\Big|
\sum_{F, G \in \mathcal{F}_{V}^{\text{cusp}}}
 \frac{b_F \overline{b_G} }{ \omega_F^{1/2} \omega_G^{1/2} }
\lambda_G(m,n) \overline{\lambda_F}(m,n)
 \Big|
 \\
 \ll \frac{N}{T^3} (NT)^{\varepsilon}
\max_{1 \leq Z \ll \frac{T^6}{N} (TN)^{\varepsilon}}
\Delta^{(3)}\Big(\mathcal{F}_1, Z \Big). \qedhere
\end{multline}
\end{proof}

\section{Completion of the proof}
Now we prove Theorem \ref{thm:mainthm}.  We chain together the results from Section \ref{section:combinatorial} as well as Lemma \ref{lemma:FunctionalEquation}, giving
\begin{multline}
\label{eq:chainofinequalities}
\Delta_1(\mathcal{F}_1, N) \leq
\Delta^{(3)}(\mathcal{F}_1, N)
\ll 
N + 
\frac{N}{T^3} (NT)^{\varepsilon}
\max_{1 \leq Z \ll \frac{T^6}{N} (TN)^{\varepsilon}}
\Delta^{(3)}\Big(\mathcal{F}_1, Z \Big)
\\
\ll
N + 
\frac{N}{T^3} (NT)^{\varepsilon}
\max_{Y^2 X \ll \frac{T^6}{N} (NT)^{\varepsilon}} 
\Big(\frac{T^6/N}{XY^2}\Big)^{1/3} 
\min\Big( Y \Delta_1(\mathcal{F}_1, X), X \Delta_1(\mathcal{F}_1, Y) \Big).
\end{multline}
This sequence of inequalities is reminiscent of (and somewhat inspired by) \cite[Section 8]{HeathBrownCubicSieve}.  Finally we insert the Blomer-Buttcane bound $\Delta_1(\mathcal{F}_1, M) \ll (T^3 + T^2 M) (TM)^{\varepsilon}$ from Theorem \ref{thm:BlomerButtcane}.  In all, we obtain
\begin{align*}
\Delta_1(\mathcal{F}_1, N)
&\ll 
N + 
\frac{N}{T^3} (NT)^{\varepsilon}
\max_{Y^2 X \leq \frac{T^6}{N}} 
\Big(\frac{T^6/N}{XY^2}\Big)^{1/3} 
\min\Big( Y (T^3 + T^2 X), X (T^3 + T^2 Y) \Big)
\\
&\ll
N + 
\frac{N}{T^3}  (NT)^{\varepsilon}
\max_{Y^2 X \leq \frac{T^6}{N}} 
\Big(\frac{T^6/N}{XY^2}\Big)^{1/3} 
T^2
\Big(XY + T \min(X,Y) \Big)
\\
&\ll 
N + \frac{N}{T^3} T^2 \Big(\frac{T^6}{N} + T \Big(\frac{T^6}{N}\Big)^{1/3}\Big) (NT)^{\varepsilon}
\ll N + (NT)^{\varepsilon} ((T^3 N)^{2/3} + T^5 ),
\end{align*}
completing the proof of Theorem \ref{thm:mainthm}.
%Simplifying the above bound is a bit tedious but can be made easier by noting that the function being maximized (namely, $\frac{XY + T \min(X,Y)}{(XY^2)^{1/3}}$ is homogeneous in $X,Y$.  

\section{Cubic characters}
\label{section:cubic}
In this section we briefly recall the large sieve inequality of Heath-Brown for cubic characters \cite{HeathBrownCubicSieve} for the purpose of developing an analogy with the $SL_3(\mathbb{Z})$ cusp form family considered in this paper.

Let $\theta = \exp(2 \pi i/3)$, and for nonzero $m,n \in \mz[\theta]$ let $(m/n)_3$ denote the cubic residue symbol.  The cubic reciprocity law gives that $(m/n)_3 = (n/m)_3$.
Let $N(\cdot)$ denote the norm map of $\mathbb{Q}[\omega]/\mathbb{Q}$.  Let
\begin{equation}
\Delta_1(M,Q) = \max_{|{\bf a}| = 1}  \sumstar_{\substack{n \in \mz[\omega] \\ N(n) \leq M}}
\Big| \sumstar_{\substack{q \in \mz[\theta] \\ N(q) \leq Q}} a_q \Big(\frac{n}{q}\Big)_3 \Big|^2,
\end{equation}
where the symbol $\sum^{*}$ means the sums are restricted to (nonzero) square-free integers.
Heath-Brown's cubic large sieve is the bound
\begin{equation}
\label{eq:HB1}
\Delta_1(M,Q)
\ll (M + Q + (MQ)^{2/3})(MQ)^{\varepsilon}.
\end{equation}

To make the notation appear more similar to the $SL_3(\mathbb{Z})$ family, define (for $m,n,q \in \mz[\theta]$)
\begin{equation}
\lambda_q(m,n) = \Big(\frac{n}{q}\Big)_3 \overline{\Big(\frac{m}{q}\Big)_3}.
\end{equation}
Note the simple identities which the reader is invited to compare with Lemma \ref{lemma:Hecke}:
$$\lambda_q(m,n) = \overline{\lambda_q(n,m)} = \lambda_q(mn^2, 1) = \lambda_q(1,nm^2) = \lambda_q(m,1) \lambda_q(1,n).$$
Also observe $\lambda_q(d^3, 1) = 1$ for $(d,q) = 1$.

Heath-Brown's first step is to drop the condition that $n$ is square-free in \eqref{eq:HB1}, leading to the definition
\begin{equation}
\Delta_3(M,Q) = \max_{|{\bf a}| = 1}  \sum_{\substack{d,m,n \in \mz[\theta] \\ N(d^3 m^2 n) \leq M}} |\mu(mn)|
\Big| \sumstar_{\substack{q \in \mz[\theta] \\ N(q) \leq Q}} a_q \lambda_q(m,n) \Big|^2
\end{equation}
Obviously $\Delta_1(M,Q) \leq \Delta_3(M,Q)$, which parallels our relation $\Delta_1(\mathcal{F}_{V}, N) \leq \Delta_3(\mathcal{F}_{V}, N)$.  The same steps used to prove Lemma \ref{lemma:Delta3intermsofDelta1} can be applied here to show
\begin{equation}
\Delta_3(M,Q) \ll (MQ)^{\varepsilon}
\max_{XY^2 \ll M} \Big(\frac{M}{XY^2}\Big)^{1/3} 
\min(X \Delta_1(Y, Q), Y \Delta_1(X,Q)),
\end{equation}
which is essentially \cite[Lemma 6]{HeathBrownCubicSieve}.

Heath-Brown also gives a relationship between $\Delta_3(M,Q)$ and $\Delta_3(Q^2/M, Q)$ (see \cite[Lemmas 7 and 8]{HeathBrownCubicSieve} for the precise statement) which arises from the functional equation and is analogous to Lemma \ref{lemma:FunctionalEquation}.

\section{Lower bound via duality}
\label{section:lowerbound}
\begin{myprop}
\label{prop:lowerbound}
There exists a choice of vector ${\bf a}$ so that
\begin{equation}
\sum_{F \in \mathcal{F}_1^{\text{Eis}}} \frac{1}{\omega_F}
\Big| \sum_{N \leq n \leq 2N} a_n \lambda_F(1,n)\Big|^2 \gg (TN)^{1-\varepsilon} |{\bf a}|^2,
\end{equation}
for $N \gg T^{7/3 + \delta}$.
\end{myprop}
Since a lower bound of this type was already proved in \cite{BlomerButtcane}, for brevity we only give a sketch which could be made rigorous with more work.
Blomer and Buttcane \cite[Section 4]{BlomerButtcane} showed that the lower bound of size ``$T^2 N$'' in \eqref{eq:BlomerButtcaneWide} comes from the Eisenstein series $E(z, 1/2+it, u_j)$ induced from $SL_2(\mathbb{Z})$ cusp forms $u_j$.
This Eisenstein series $E$ has Hecke eigenvalues
$$\lambda_E(1,n) = \lambda(n) = \sum_{d_1 d_2 = n} \lambda_j(d_1) d_1^{-it} d_2^{2it},$$
and Langlands parameters $\mu = (2it, -it+it_j, -it-it_j)$, where $t_j$ is the spectral parameter of $u_j$.  Moreover, $\omega_F = T^{o(1)}$, so we will drop this aspect in the proof.

\begin{proof}
To simplify notation, say that $B$ is the spectral ball of size $O(1)$ centered at $i(2T, T, -3T)$.  This means $t = T+ O(1)$ and $t_j = T + O(1)$.  
The contribution of this family of Eisenstein series, on the dual side, takes the form (after smoothing)
\begin{equation}
\label{eq:EisensteinExpression}
\mathcal{S} := \sum_{n} w(n/N) 
\Big|
\int_{t,t_j = T + O(1)} \beta_{t,t_j} \lambda(n) \Big|^2.
\end{equation}
Expanding the square, we obtain
\begin{equation}
\label{eq:EisensteinExpression2}
\mathcal{S} = 
\int_{t,t', t_j, t_j' = T + O(1)}
\beta \overline{\beta'} \frac{1}{2 \pi i}
\int_{(1+\varepsilon)} N^s \widetilde{w}(s) 
\underbrace{\sum_{n=1}^{\infty} \frac{\lambda(n) \overline{\lambda'}(n)}{n^s}}_{Z_{u_j,u_j', t, t'}(s)} ds.
\end{equation}
%The inner Dirichlet series above, i.e., $Z_{u_j,u_j', t, t'}(s)$, takes the form
Note
\begin{multline}
\label{eq:DirichletSeriesEisCase}
Z_{u_j,u_j', t, t'}(s) = \sum_{d_1 d_2 = e_1 e_2}
\frac{\lambda_j(d_1) d_1^{-it} d_2^{2it} \lambda_j'(e_1) e_1^{it'} e_2^{-2it'} }{(d_1 d_2)^s}
\\
=
\prod_p (1 + 
p^{-s}[\lambda_j(p) \lambda_{j}'(p) p^{-it+it'} + \lambda_j(p) p^{-it-2it'} +  \lambda_j'(p) p^{2it+it'} + p^{2it-2it'}] + O_{u_j, u_j'}(p^{-2s}) ).
\end{multline}
With some care, including use of the convexity bound for $GL_2 \times GL_2$ $L$-functions due to Iwaniec \cite{IwaniecConvexity}, one may then derive
\begin{equation}
\label{eq:Lfunctions}
Z_{u_j,u_j', t, t'}(s) = 
\zeta(s-2it+2it') L(s+it+2it', u_j) L(s-2it-it', u_j') L(s+it-it', u_j \otimes u_j')
A(s),
\end{equation}
where $A(s) = A_{u_j, u_j', t, t'}(s)$ is given by an absolutely convergent Euler product for $\text{Re}(s) > 1/2$, satisfying $|A(\sigma+iy)| \ll_{\sigma} T^{\varepsilon}$, for $\sigma > 1/2$.

Returning to \eqref{eq:EisensteinExpression2}, we shift contours to the line $\text{Re}(s) = 1/2+\varepsilon$.  
Note the pole of zeta at $s=1+2it-2it'$ which occurs for all pairs $u_j, u_j'$.  
This polar term, say denoted $\mathcal{S}_0$, contributes (roughly)
\begin{equation}
\int_{t,t', t_j, t_j' = T + O(1)}
\beta \beta'  N^{1+2it-2it'} \widetilde{w}(1+2it-2it') 
L(1+3it, u_j) L(1-3it', u_j') L(1+3it-3it', u_j \otimes u_j').
\end{equation}
If we choose $\beta_{t,t_j} = L(1+3it, u_j)^{-1}$ (alternatively, one could take $\beta_{t,t_j} = \overline{L(1+3it, u_j)}$) then this polar term becomes approximately
\begin{equation}
N \sum_{t_j, t_j' = T+O(1)} L(1, u_j \otimes u_j') \approx NT^2.
\end{equation}
With a bit more care, one can derive $|\mathcal{S}_0| \gg (NT^2)^{1-\varepsilon}$ for this choice of $\beta$.  
  
Next we estimate the contribution to $\mathcal{S}$ from the new line of integration at $\text{Re}(s) = 1/2+\varepsilon$; call this $\mathcal{S}'$.  
Jutila and Motohashi \cite{JM} showed a Weyl bound for the $SL_2(\mathbb{Z})$ cusp forms,
namely
\begin{equation}
L(1/2+it, u_j) \ll (1 + |t| + |t_j|)^{1/3+\varepsilon}.
\end{equation}
 Combining this with the convexity bound for the $GL_2 \times GL_2$ factor in \eqref{eq:Lfunctions} (which has conductor of size $T^2$) gives $|Z(\sigma+iy)| \ll T^{7/6+\varepsilon}$.  
Note that the $\zeta$-factor is evaluated at $\sigma + i y$ with $|y| \ll T^{\varepsilon}$, so it practically gives no contribution here.
Hence, for this choice of $\beta$, we have
\begin{equation}
\mathcal{S}' \ll N^{1/2+\varepsilon} T^{7/6+\varepsilon} 
\int_{t,t',t_j,t_j' = T+O(1)} |\beta \beta'| \ll N^{1/2+\varepsilon} T^{7/6+\varepsilon} T^2.
\end{equation}
Thus
\begin{equation}
|\mathcal{S}| \gg (NT^2)^{1-\varepsilon} + O(N^{1/2+\varepsilon} T^{7/6+\varepsilon} T^2).
\end{equation}
Note the polar term dominates the error term provided $N \gg T^{7/3+\delta}$.
Finally, we observe that
\begin{equation}
\int_{t, t_j = T+O(1)} |\beta_{t,t_j}|^2 dt =  T^{1+o(1)},
\end{equation}
whence $|\mathcal{S}| \gg (NT)^{1-\varepsilon} \int |\beta|^2$ with this choice of $\beta$.
\end{proof}

\section{Loose ends}
We list a few possible directions for future work.
\begin{enumerate}
 \item Extend Theorem \ref{thm:mainthm} to cover the family $\mathcal{F}_{V}$ for more general $V$, with $1 \ll V \ll T$.
 \item Extend Theorem \ref{thm:BlomerButtcane}, which gives a bound on the norm $\Delta_1$ using the Kuznetsov formula, to directly give a bound on the norm $\Delta_2$ (modified to include the Eisenstein series as well as the cusp forms).
% \begin{equation}
% \max_{|{\bf a}| = 1} \sum_{* \in \{ \text{cusp}, \text{Eis} \}}
% \sum_{F \in \mathcal{F}^{*}} \frac{1}{\omega_F} 
%  \Big| \sum_{N \leq m^2 n \leq 2N} a_{m,n} \lambda_F(m,n) \Big|^2.
%  \end{equation}
  The point would be to bypass the use of Lemma \ref{lemma:Delta2intermsofDelta1} in \eqref{eq:chainofinequalities}, though
 it is unclear if any improvement is possible this way.
 \item Is it possible to use the $SL_3(\mathbb{Z})$ Kuznetsov formula to directly bound the cuspidal part of the spectrum in the large sieve inequality? (By subtracting off the Eisenstein parts, which one would presumably then control with lower-rank tools such as the $GL_2$ Kuznetsov formula.)
 \item Can the term $T^2 N$ in \eqref{eq:BlomerButtcaneLocal} be reduced to $TN$, to match the lower bound from Proposition \ref{prop:lowerbound}?
 If so, this would immediately improve Theorem \ref{thm:mainthm} by replacing the term $T^5$ by $T^4$.
\end{enumerate}

\end{document}